\documentclass{article}
\usepackage{enumerate}
\usepackage[hyperindex,breaklinks,colorlinks,citecolor=blue]{hyperref}
\usepackage[latin1]{inputenc}
\usepackage[thicklines]{cancel}
\usepackage{amsfonts, amsmath, amssymb, amsthm}
\usepackage{mathrsfs}
\usepackage{graphicx}
\usepackage{latexsym}
\usepackage[all]{xy}

\newtheorem{lemma}{Lemma}[section]
\newtheorem{proposition}{Proposition}[section]
\newtheorem{theorem}{Theorem}[section]
\newtheorem{corollary}{Corollary}[section]

\theoremstyle{definition}
\newtheorem{definition}{Definition}[section]

\theoremstyle{remark}
\newtheorem{remark}{Remark}[section]
\newtheorem{example}{Example}[section]

\newcommand{\defin}[1]{\emph{\textbf{#1}}}
\newcommand{\bd}[1]{\boldsymbol{#1}}

\newcommand{\A}{\mathbf{A}}

\newcommand{\C}{\mathcal{C}}

\newcommand{\N}{\mathbb{N}}
\renewcommand{\O}{\mathcal{O}}

\newcommand{\R}{\mathbb{R}}

\newcommand{\sierp}{\mathbb{S}}
\newcommand{\U}{\mathcal{U}}
\newcommand{\V}{\mathcal{V}}
\newcommand{\W}{\mathcal{W}}
\newcommand{\X}{\mathbf{X}}
\newcommand{\Y}{\mathbf{Y}}
\newcommand{\Z}{\mathbf{Z}}

\newcommand{\prodtimes}{\mathrel{\times_\text{prod}}}
\newcommand{\id}{\text{id}}
\newcommand{\eval}{\text{eval}}

\newcommand{\Tot}{\text{Tot}}
\newcommand{\dom}{\text{dom}}
\newcommand{\SSigma}{\pmb{\boldsymbol{\Sigma}}}
\newcommand{\PPi}{\pmb{\boldsymbol{\Pi}}}

\newcommand{\QCB}[1]{\mathsf{QCB}_0(\PPi^1_{#1})}
\newcommand{\Based}{\text{Based}}

\newcommand{\KK}[1]{\N\uple{#1}}
\newcommand{\NN}{\KK{2}}
\newcommand{\Baire}{\mathcal{N}}
\newcommand{\Bairec}{\Baire_c}
\newcommand{\cantor}{2^\N}

\newcommand{\uple}[1]{{\langle #1 \rangle}}

\newcommand{\cf}{\bd{1}} 

\newcommand{\res}[1]{\mathop{\upharpoonright_{#1}}}

\newcommand{\tto}{\rightrightarrows}

\author{Mathieu Hoyrup\footnote{615 rue du jardin botanique, 54600 Villers-l\`es-Nancy, France, mathieu.hoyrup@inria.fr}}

\title{Results in descriptive set theory on some represented spaces}

\begin{document}
\maketitle

%
%
%
%
%
%

Descriptive set theory was originally developed on Polish spaces. It was later extended to~$\omega$-continuous domains \cite{Selivanov04} and recently to quasi-Polish spaces \cite{Brecht13}. All these spaces are countably-based. Extending descriptive set theory and its effective counterpart to general represented spaces, including non-countably-based spaces has been started in \cite{PaulyB15}.

We study the spaces~$\O(\N^\N)$,~$\C(\N^\N,2)$ and the Kleene-Kreisel spaces~$\KK{\alpha}$. We show that there is a~$\Sigma^0_2$-subset of~$\O(\N^\N)$ which is not Borel. We show that the open subsets of~$\N^{\N^\N}$ cannot be continuously indexed by elements of~$\N^\N$ or even~$\N^{\N^\N}$, and more generally that the open subsets of~$\KK{\alpha}$ cannot be continuously indexed by elements of~$\KK{\alpha}$. We also derive effective versions of these results.

These results give answers to recent open questions on the classification of spaces in terms of their base-complexity, introduced in \cite{deBSS16}. In order to obtain these results, we develop general techniques which are refinements of Cantor's diagonal argument involving multi-valued fixed-point free functions and that are interesting on their own right.

\section{Introduction}
The core concepts of descriptive set theory are the definitions of pointclasses, i.e.\ classes of sets, according to the way those sets can be described.

It was originally developed on Polish spaces, where the definitions of pointclasses are extensional: they describe how the sets in the classes~$\SSigma_0^1$,~$\Sigma^0_1$,~$\SSigma^0_2$, etc.\ are built starting from basic open sets in a countable basis \cite{Kechris95}. In particular these definitions imply that all these classes can be parametrized by the Baire space~$\Baire=\N^\N$, which is usually expressed by saying that on a Polish space~$\X$, each class~$\Gamma(\X)$ has a~$\Baire$-universal set, i.e.\ a set~$U$ in~$\Gamma(\Baire\times \X)$ such that~$\Gamma(\X)=\{U_f:f\in\Baire\}$ where~$U_f=\{x\in\X:(f,x)\in U\}$. All this extends to quasi-Polish space, by slightly changing the definitions of the pointclasses \cite{Selivanov04,Brecht13}.

On a general represented space, the definitions are intensional: a set belongs to a pointclass if its preimage under the representation belongs to that pointclass \cite{PaulyB15}. However this definition does not give information about what those sets look like and how to obtain them.


Given a represented space, we are interested in understanding more concretely what the elements of the various pointclasses look like, or how they can be described. We provide here essentially impossibility results, stating that there is no simple way of describing them. A typical result is that for some particular represented spaces~$\X,\Y$, there is no~$\Y$-universal open set, i.e.\ there is no open subset~$U$ of~$\X\times\Y$ such that~$\O(\X)=\{U_y:y\in\Y\}$ or equivalently there is no continuous surjection from~$\Y$ to~$\O(\X)$. The space~$\Y$ is usually~$\Baire$ or some Kleene-Kreisel space~$\KK{\alpha}$. We also derive effective versions of such a result: there is no computable surjection from any~$\Sigma^1_k$-subset of~$\N$ to the computable open subsets of~$\X$, or there is no~$\Delta^1_k$-oracle enumerating those computable open sets.

We develop general tools that enable to prove these impossibility results in a number of interesting cases, in particular for the space of clopen subsets of~$\Baire$ or the Kleene-Kreisel spaces~$\KK{\alpha}$. These tools treat the continuous and computable settings in the same way.

Our tools and techniques are based on Cantor's diagonal argument: the existence of a fixed-point free function is used to disprove the existence of an enumeration. However we usually work on spaces that are lattices of open sets, where every function has a fixed-point but where a \emph{multi-valued} fixed-point free function may exist, and we need to develop techniques enabling us to deal with multi-valued functions. In particular we introduce the notion of a \emph{canonizable space}.

%
%
%
%
%

The paper is organized as follows. In Section \ref{sec_background} we give all the necessary notions about represented spaces and their topologies. In Section \ref{sec_nonborel} we prove that the space~$\O(\Baire)$ contains a~$\Sigma^0_2$-subset that is not Borel. In Section \ref{sec_param} we develop techniques based on the diagonal argument to prove some impossibility results. In Section \ref{sec_applications} we apply these techniques to prove that some particular spaces have no~$\Baire$-universal open set. In Section \ref{sec_canon} we refine our techniques and introduce the notion of a canonizable space to include parameters spaces other than~$\Baire$ and its subspaces. We apply this technique to classify the Kleene-Kreisel spaces and and to prove a result about Markov computability.

\section{Background}\label{sec_background}
Represented spaces are at the basis of computability theory on general spaces \cite{Wei00}. We follow the modern presentation given in \cite{Pauly15}, \cite{PaulyB15}.

A \defin{represented space} is a pair~$\X=(X,\delta_X)$, where~$\delta_X:\subseteq\Baire\to X$ is a partial onto map. We often identify~$\X$ with its carrier set~$X$, for commodity. If~$\delta_X(p)=x$ then we say that~$p$ is a \defin{name} of~$x$.  A point~$x\in\X$ is \defin{computable} if it has a computable name. We denote by~$\X_c$ the set of computable points of~$\X$. A function~$f:\X\to\Y$ is \defin{continuous} (resp.\ \defin{computable}) if there exists a continuous (resp.\ computable) function~$F:\dom(\delta_X)\to \dom(\delta_Y)$ mapping any name of any~$x\in\X$ to a name of~$f(x)$. 
We call~$F$ a \defin{realizer} of~$f$.

Two represented spaces~$\X,\Y$ are computably isomorphic, denoted~$\X\cong\Y$, if there is a bijection between that is computable in both directions.

\paragraph{Cartesian product.}
The product of two represented spaces~$\X,\Y$ is denoted by~$\X\times \Y$. A name of~$(x,y)$ is simply a pair of names of~$x$ and~$y$ (two elements of~$\Baire$ can be paired into another element of~$\Baire$). One has~$\X\times \Y\cong\Y\times \X$.

\paragraph{Exponentiation.}
The set of continuous functions from~$\X$ to~$\Y$ is itself a represented space denoted by~$\C(\X,\Y)$. A name for a continuous function~$f:\X\to\Y$ is a name for any continuous realizer~$F$ of~$f$ (such a name gives the information needed to evaluate~$F$, for instance it encodes a list of pairs~$(u,v)$ of finite sequences such that for every~$p\in\dom(\delta_X)$, (i) if~$u$ is a prefix of~$p$ and~$(u,v)$ appears in the list then~$F(p)$ extends~$v$ and (ii) for every prefix~$v$ of~$F(p)$ there exists a prefix~$u$ of~$p$ such that~$(u,v)$ appears in the list).

The evaluation map~$\eval:\X\times\C(\X,\Y)\to\Y$ is computable and~$\C(\X\times \Y,\Z)\cong\C(\X,\C(\Y,\Z))$.

\paragraph{Topology.}
Every represented space~$\X=(X,\delta_X)$ has a canonical topology~$\tau_X$, which is the final topology of~$\delta_X$: a set~$U\subseteq X$ is open iff~$\delta_X^{-1}(U)$ is the intersection of an open subset of~$\Baire$ with~$\dom(\delta_X)$. As~$\dom(\delta_X)\subseteq\Baire$ is countably-based, hence sequential, the final topology of~$\delta_X$ is sequential. The set of open subsets of~$\X$ is itself a represented space obtained as follows. The Sierpi\'nski space~$\sierp=(\{\bot,\top\},\delta_\sierp)$ where~$\delta_\sierp(0^\N)=\bot$ and~$\delta_\sierp(p)=\top$ for~$p\neq 0^\N$. A subset~$U$ of~$\X$ is then open exactly when its characteristic function~$\cf_U:\X\to\sierp$ is continuous. The space of open subsets of~$X$ is then~$\O(\X):=\C(\X,\sierp)$ with its canonical representation. In turn, the topology on~$\O(\X)$ is the Scott topology, which coincides with the~$\omega$-Scott topology (see \cite{deBSS16} for more details).

\paragraph{Subspace.}
If~$\X=(X,\delta_X)$ is a represented space then any subset~$A\subseteq X$ has a canonical representation~$\delta_A$, which is simply the restriction of~$\delta_X$ to~$\delta_X^{-1}(A)$. Thus it immediately induces notions of continuous and computable functions from~$A$ to any other represented space. It is known that sequential topological spaces are not stable under taking subspace topologies. Similarly, the topology on~$A$ induced by~$\delta_A$ is not in general the subspace topology, but is richer. In the same way, the topology on a product of represented spaces is not in general the product topology but is richer.

\paragraph{Admissibility.}
One has to be careful as there are two a priori distinct notions of continuity for functions from~$\X$ to~$\Y$: the functions having a continuous realizer, and the functions that are continuous w.r.t.\ to the induced topologies. One easily checks that a continuous (i.e.\ continuously realizable) function is always continuous w.r.t.\ the induced topologies, however the converse implication fails without further assumption on the representation. The appropriate condition that makes this implication an equivalence is the admissibility of the representation. A representation is admissible if, intuitively, representing a point~$x$ is equivalent to being able to know which open sets contain~$x$. The evaluation map~$\eval:\X\times\O(\X)\to\sierp$, equivalently~$\eval:\X\to\O(\O(\X))$ is computable: the space~$\X$ is \defin{admissible} (resp.\ \defin{computably admissible}) if the inverse of~$\eval:\X\to\O(\O(\X))$ is continuous (resp.\ computable). In this article, for the general results do not require the representations to be admissible, which is useful when one deals with Markov computability for instance (see Section \ref{sec_markov}).

\paragraph{Multifunctions.}
Often, a computation taking a name of~$x\in \X$ as input and producing a name of some~$y\in\Y$ as output does not induce a function from~$\X$ to~$\Y$, as~$y$ might depend not only on~$x$, but on the name of~$x$ given as input. For instance, given a real number~$x\in [0,1]$ it is possible to compute some ~$y\in [0,1]$ such that~$y\neq x$. However it is not possible to compute the same~$y$ for all name of any~$x$, as it would induce a computable, hence continuous function from~$[0,1]$ to~$[0,1]$ without fixed-point, violating the Intermediate Value Theorem. A \defin{multifunction} is a set-valued function~$f:\X\to 2^\Y$, denoted by~$f:\X\tto\Y$. A multifunction~$f:\X\tto \Y$ is continuous (resp.\ computable) if there exists a continuous (resp.\ computable) function~$F:\dom(\delta_X)\to \Baire$ such that~$\delta_Y(F(p))\in f(\delta_X(p))$ for all~$p\in\dom(\delta_X)$.

\paragraph{Descriptive set theory.}
Notions from descriptive set theory originally developed on Polish spaces \cite{Kechris95} have a straightforward extension to any represented space \cite{PaulyB15}. The question is then whether the results also extend, and how.

Any pointclass~$\Gamma(\Baire)$ on~$\Baire$ induces a pointclass on~$\X$ as follows:~$A\subseteq\X$ belongs to~$\Gamma(\X)$ if there is a set~$B\subseteq\Baire$ in~$\Gamma(\Baire)$ such that~$\delta_X^{-1}(A)=B\cap \dom(\delta_X)$.

In a topological space, the class of Borel sets is the smallest class containing the open sets and closed under taking complements and countable unions.

\subsection{Retractions}

As we have already mentioned, if~$A$ is a subset of a represented space~$\X$, then the represented subspace~$\A=(A,\delta)$ is not a topological subspace in general. Similarly, if there are continuous (i.e., continuously realizable) functions~$f:\Y\to \X$ and $g:f(\Y)\to \Y$ such that~$g\circ f=\id_\Y$, then~$f$ is not necessarily a topological embedding because~$g$ is not necessarily continuous for the subspace topology on~$f(\Y)$. However if~$g$ has a continuous extension over~$\X$ then~$f$ is a topological embedding.

\begin{definition}
A \defin{section-retraction} is a pair~$(s,r)$ of functions~$s:\Y\to\X$,~$r:\X\to \Y$ such that~$r\circ s=\id_\Y$. If both~$s$ and~$r$ are continuous, then~$\Y$ is called a \defin{continuous retract} of~$\X$. If both~$s$ and~$r$ are computable, then~$\Y$ is called a \defin{computable retract} of~$\X$.
\end{definition}
Intuitively, elements of~$\Y$ can be encoded by elements of~$\X$. If~$r(x)=y$ then~$x$ is a code for~$y$, and every~$y$ has a canonical code, given by~$s(y)$. It implies in particular that~$\Y$ is a continuous (resp.\ computable) image of~$\X$.

\begin{lemma}\label{lem_retract_open}
If~$\Y$ and~$\Y'$ are continuous (resp.\ computable) retracts of~$\X$ and~$\X'$ respectively, then~$\C(\Y,\Y')$ is a continuous (resp.\ computable) retract of~$\C(\X,\X')$. In particular~$\O(\Y)$ is a continuous (resp.\ computable) retract of~$\O(\X)$.
\end{lemma}
\begin{proof}
Let~$(s,r)$ and~$(s',r')$ be sections-retractions for~$\Y,\X$ and~$\Y',\X'$ respectively. Let~$S:\C(\Y,\Y')\to\C(\X,\X')$ map~$f$ to~$s'\circ f\circ r$ and~$R:\C(\Y,\Y')\to\C(\X,\X')$ map~$g$ to~$r'\circ g\circ s$. The pair~$(S,R)$ is easily a section-retraction. The case of~$\O(\Y),\O(\X)$ is obtained by taking~$\X'=\Y'=\sierp$.
\end{proof}

%
%

If~$A$ is a \emph{closed} subset of~$X$ then the topology of the represented space~$\A$ coincides with the subspace topology. Moreover,
\begin{lemma}\label{lem_retract_closed}
If~$A$ is a closed subset of~$\X$ then~$\O(\A)$ is a continuous retract of~$\O(\X)$. If~$A$ is a~$\Pi^0_1$-subset of~$\X$ then~$\O(\A)$ is a computable retract of~$\O(\X)$.
\end{lemma}
\begin{proof}
Let~$r:\O(\X)\to\O(\A)$ map~$V$ to~$V\cap A$ and~$s:\O(\A)\to\O(\X)$ map~$U$ to~$U\cup A^c$.
\end{proof}


We will need the following result in several places.
\begin{lemma}\label{lem_retract_NN_NN2}
$\Baire$ is a computable retract of~$\C(\Baire,2)$.
\end{lemma}
\begin{proof}
Each clopen set~$C\in\C(\Baire,2)$ encodes a function~$f:\N\to\N$, which is a kind of modulus of continuity of~$C$ at certain points, defined as follows. For each~$n$, let~$c_n:\N\to\N$ be the constant function with value~$n$, and~$c_{n,k}(i)=n$ for~$i\neq k$,~$c_{n,k}(k)=n+1$. If~$C$ does not contain~$c_n$, then let~$f(n)=0$. If~$C$ contains~$c_n$, then let~$f(n)=\max\{k\geq 1:c_{n,k}\notin C\}$. We then define~$r(C):=f$.

Every~$f:\N\to\N$ is encoded in some~$C\in\C(\Baire,2)$, defined as follows. For each~$n\in\N$, if~$f(n)=0$ then let~$C\cap [n]=\emptyset$. If~$f(n)=k+1$ then let~$C\cap[n]=\{g:g(k+1)=n\}$. We then define~$s(f):=C$.

The functions~$r$ and~$s$ are computable and one easily checks that~$r\circ s(f)=f$ for all~$f$.
\end{proof}

\section{A non-Borel set}\label{sec_nonborel}
On Polish spaces, the Borel hierarchy built from the open sets coincides with the hierarchy lifted from~$\Baire$ by the representation: for instance a set is a countable union of closed sets if and only if its pre-image is a countable union of closed sets.

On quasi-Polish spaces, the same holds if the definition of the Borel hierarchy is slightly amended: for instance~$\SSigma^0_2$-sets are not countable unions of closed sets, but countable unions of differences of open sets.

Here we show that on admissibly represented spaces this is no more true in general: in the space~$\O(\Baire)$ there exists a~$\Sigma^0_2$-subset of~$\O(\Baire)$ that is not even Borel.

To build this set, we first work on an intermediate space.
\subsection{Product topology}
We have seen that two represented spaces~$\X$ and~$\Y$ naturally induce a third represented space~$\X\times \Y$. The topology induced by that representation is not in general the product topology, but its sequentialization.

A simple example is given by~$\X=\Baire$ and~$\Y=\O(\Baire)$. The evaluation map~$\Baire\times\O(\Baire)\to\sierp$ is continuous (and computable), however it is not continuous w.r.t.\ the product topology, because~$\Baire$ is not locally compact (see \cite{EscardoH01} for more details on this topic). In other words the set~$\{(f,O)\in\Baire\times\O(\Baire):f\in O\}$ is not open for the product topology (but it is sequentially open, or open for the topology induced by the representation). It is even worse.
\begin{proposition}\label{prop_not_borel}
$E=\{(f,O)\in\Baire\times\O(\Baire):f\in O\}$ is not Borel for the product topology.
\end{proposition}
\begin{proof}
We prove that for every Borel set~$A$, there exists a dense~$G_\delta$-set~$G\subseteq\Baire$ such that for every~$f\in G$,~$(f,\Baire\setminus\{f\})\in A\iff (f,\Baire)\in A$. It implies the result as it is obviously false for the set~$E$. To prove it, we show that the class of sets satisfying this condition contains the open sets in the product topology and is closed under taking complements and countable unions, which implies that this class contains the Borel sets.

First, consider a basic open set~$A=[u]\times \U_K$ where~$u$ is a finite sequence of natural numbers,~$K$ is a compact subset of~$\Baire$ and~$\U_K=\{O\in\O(\Baire):K\subseteq O\}$. Define~$G=[u]^c\cup [u]\setminus K$, which is a dense open set. For~$f\in[u]^c$, no~$(f,O)$ belongs to~$A$. For~$f\in[u]\setminus K$, both~$(f,\Baire\setminus \{f\})$ and~$(f,\Baire)$ belong to~$A$.

If~$A$ satisfies the condition with a dense~$G_\delta$-set~$G$, then~$A^c$ satisfies the condition with the same~$G$. If~$A_i$ satisfy the condition with dense~$G_\delta$-sets~$G_i$ then~$\bigcup_i A_i$ satisfies the condition with~$G=\bigcap_i G_i$.
\end{proof}


\subsection{The space of open subsets of the Baire space}
We can now prove the result. Note that in the statement, the class of Borel subsets of~$\O(\Baire)$ is as usual the smallest class containing the open subsets and closed under taking complement and countable intersections.
\begin{theorem}\label{thm_sigma2}
There is a~$\Sigma^0_2$-subset of~$\O(\Baire)$ which is not Borel.
\end{theorem}
\begin{proof}
We show that~$\Baire\prodtimes \O(\Baire)$ is a retract of~$\O(\Baire)$ in some way. We build: 
\begin{itemize}
\item A continuous function~$s:\Baire\prodtimes\O(\Baire)\to\O(\Baire)$,
\item A~$\Sigma^0_2$-measurable function~$r:\O(\Baire)\to\Baire\times\O(\Baire)$,
\item Such that~$r\circ s=\id$.
\end{itemize}
First, these ingredients enable to derive the result. Indeed, let~$E$ from Proposition \ref{prop_not_borel} and~$F:=r^{-1}(E)\subseteq\O(\Baire)$. As~$E$ is open in~$\Baire\times\O(\Baire)$,~$F$ is~$\Sigma^0_2$. However~$F$ is not Borel, otherwise~$E=s^{-1}(F)$ would be Borel in~$\Baire\prodtimes\O(\Baire)$.

Let us now build~$s$ and~$r$. First observe that~$\O(\Baire)\cong \O(\Baire)\times\O(\Baire)$ and that the topology on~$\O(\Baire)\times\O(\Baire)$ coincides with the product topology.

Lemma \ref{lem_retract_NN_NN2} provides two computable functions~$r_0:\C(\Baire,2)\to \Baire$ and~$s_0:\Baire\to\C(\Baire,2)$ such that~$r_0\circ s_0=\id$. Observing that~$\sierp$ and~$2$ are two represented spaces with the same underlying set and that~$\O(\Baire)\cong\C(\Baire,\sierp)$, those functions can be interpreted as~$r_1:\O(\Baire)\to\Baire$ and~$s_1:\Baire\to\O(\Baire)$ with~$r_1\circ s_1=\id$. However, while~$s_1$ is still computable,~$r_1$ is not computable but is~$\Sigma^0_2$-measurable.

Let us simply pair~$s_1$ and~$r_1$ with the identity on~$\O(\Baire)$: let~$s(f,O')=(s_1(f),O'))$ and~$r(O,O')=(r_1(O),O')$. 
%
%
\end{proof}

In particular, that~$\Sigma^0_2$-set is not a countable union of differences of open sets, as it should be on Polish or quasi-Polish spaces. More generally, it is not a countable boolean combination of open sets.

In order to overcome the mismatch between the hierachy inherited from~$\Baire$ via the representation and the class of Borel sets, one may attempt to change the definition of Borel sets. In \cite{NorbergV97} the Borel sets are redefined as the smallest class containing the open sets and the saturated compact sets. We observe here that this class is too large in the space~$\O(\Baire)$. First, if~$U\subseteq\Baire$ is open then the set~$\{V\in\O(\Baire):U\subseteq V\}$ is compact and saturated in~$\O(\Baire)$. From this it is easy to see that the set built above is Borel in this weaker sense. However this notion of Borel sets is too loose, because compact saturated sets do not usually have a Borel pre-image. For instance, the singleton~$\{\Baire\}$ is compact saturated but its pre-image is a~$\PPi^1_1$-complete set, hence is not Borel.

We leave the following questions for future work:
\begin{itemize}
\item What do the~$\SSigma^0_2$-subsets of~$\O(\Baire)$ look like?
\item Is it possible to modify the definition of Borel sets on~$\O(\Baire)$ to match exactly the sets that have a Borel pre-image under the representation?
\end{itemize}

\section{Parametrizations}\label{sec_param}
In this section we investigate whether a collection of objects~$A$ can be parametrized by another collection~$B$. Formally,~$A$ and~$B$ are represented spaces and~$A$ is parametrized by~$B$ if~$A$ is a continuous image of~$B$. We will also be interested in the computable version.



Given~$A$, one tries to find the ``simplest'' space~$B$ such that~$A$ is a continuous image of~$B$. Simplicity can be measured by restricting our attention to spaces~$B$ that are subspaces of~$\Baire$, and identifying their minimal descriptive complexity. The Borel hierachy is too fine-grained for this: $A$ is a continuous image of some Borel subset of~$\Baire$ iff it a continuous image of a~$\Sigma^1_1$-subset of~$\Baire$ iff it is a continuous image of~$\Baire$. Therefore, the complexity of parametrizing a space is better measured by the hyperprojective hierarchy~$\{\SSigma^1_\alpha\}$ where~$\alpha$ ranges over the countable ordinals. It happens that the levels of this hierarchy correspond to a hierarchy of spaces, the Kleene-Kreisel spaces~$\KK{\alpha}$, as stated by Theorem \ref{thm_image_kk} below: 
a represented space is a continuous image of some~$\SSigma^1_\alpha$-subset of~$\Baire$ if and only if it is a continuous image of~$\KK{\alpha}$. 

\subsection{Kleene-Kreisel functionals}
The Kleene-Kreisel functionals were introduced independently by Kleene \cite{kleene1959countable} and Kreisel \cite{Kreisel59}. They can be defined in the framework of represented spaces, as in \cite{SchroderS15b} for instance.

If~$\alpha$ is a countable ordinal then the space~$\KK{\alpha}$ is defined inductively as follows. Let~$\KK{0}=\N$,~$\KK{\alpha+1}=\C(\KK{\alpha},\N)$ and~$\KK{\lambda}=\prod_{i\in\omega}\KK{\beta(\lambda,i)}$ for a limit ordinal~$\lambda$ and a fixed numbering~$\beta(\lambda,i)$ of the ordinals smaller than~$\lambda$.

We will need the following results.

\paragraph{Functionals of finite type.}
The following results can be found in \cite{Normann80} and were proved by Kleene and Kreisel.
\begin{lemma}\label{lem_kk_dense}
For each~$k\in\N$, $\KK{k}$ contains a dense computable sequence.
\end{lemma}

\begin{theorem}[Theorem 5.22 in \cite{Normann80}]\label{thm_kk_sigma}
Let~$k\in\N$ and~$A\subseteq\Baire$. The following conditions are equivalent:
\begin{itemize}
\item $A\in\Sigma^1_k(\Baire)$,
\item There exists a computable predicate~$R\subseteq\Baire\times\KK{\alpha}\times\N$ such that
\begin{equation*}
x\in A\iff\exists y\in\KK{\alpha}\forall n,R(x,y,n).
\end{equation*}
\end{itemize}
\end{theorem}
In particular, every~$\Sigma^1_k$-set is a computable image of a~$\Pi^0_1$-subset of~$\KK{k}$, but Theorem \ref{thm_kk_sigma} says more and we will need it later.

\paragraph{Functionals of countable type.}
In \cite{SchroderS15a} an admissible representation~$\delta_\alpha$ of~$\KK{\alpha}$ is built. Let~$D_\alpha=\dom(\delta_\alpha)$.

\begin{theorem}[\cite{SchroderS15a}]\label{thm_image_kk}
Let~$\alpha$ be a countable ordinal. One has~$D_{\alpha+1}\in\PPi^1_\alpha(\Baire)$. A set~$A\subseteq\Baire$ belongs to~$\SSigma^1_\alpha(\Baire)$ iff it is a continuous image of~$\KK{\alpha}$.
\end{theorem}
\subsection{Diagonal argument.}
We now present the basic tool used to prove impossibility results about continuous parametrizations. It is a variation on Cantor's diagonal argument. It might be possible to make it an instance of Lawvere's categorical formulation of the diagonal argument \cite{Lawvere69}.

Let~$\Y$ be a represented space. If~$h:\Y\tto \Y$ is a multifunction then a point~$y\in Y$ is a \defin{fixed-point} of~$h$ if~$y\in h(y)$. We say that~$h$ is \defin{fixed-point free} if~$h$ has no fixed-point. A computable fixed-point free multifunction can be thought as an algorithm taking some~$y$ as input (represented by a name) and producing some~$y'\neq y$ as output (which may not be the same for all names of~$y$).

\begin{proposition}\label{prop_continuous}
Let~$\Y$ be a represented space. If~$\Y$ admits a continuous fixed-point free multifunction, then~$\C(\Baire,\Y)$ is not a continuous image of~$\Baire$.
\end{proposition}
\begin{proof}
Let~$h:\Y\tto \Y$ have no fixed-point and~$H:\dom(\delta_Y)\to\dom(\delta_Y)$ be a continuous realizer of~$h$. Let~$\phi:\Baire\to\C(\Baire, \Y)$ be continuous. The function~$\phi$ can equivalently be typed as~$\phi:\Baire\times\Baire\to\Y$, let~$\Phi:\Baire\times\Baire\to\dom(\delta_Y)$ be a continuous realizer of~$\phi$. We define~$F:\Baire\to\Y$ by~$F(f):=\delta_Y(H(\Phi(f)(f)))$. The function~$F$ is continuous and is not in the range of~$\phi$ because for each~$f$,~$F$ differs from~$\phi(f)$ at~$f$:~$F(f)=\delta_Y(H(\Phi(f)(f)))\neq \delta_Y(\Phi(f)(f))=\phi(f)(f)$. Therefore~$\phi$ is not onto.
\end{proof}

In other words, in that case~$\C(\Baire,\Y)$ is not a continuous image of any~$\SSigma^1_1$-subset of~$\Baire$.
\begin{remark}
Under the same assumption about~$\Y$, the same argument shows that for any set~$A\subseteq\Baire$,~$\C(A,\Y)$ is not a continuous image of~$A$.
\end{remark}

\subsection{Complexity of computable enumerations}
We will also investigate the computable counterpart of parametrizations.

\begin{definition}
Let~$\X$ be a represented space and~$\Gamma(\N)$ a pointclass. We say that~$\X_c$ is \defin{$\Gamma$-enumerable} if~$\X_c$ is a computable image of some set in~$\Gamma(\N)$.
\end{definition}

The proof of Proposition \ref{prop_continuous} is effective and immediately implies the following effective version.
\begin{proposition}\label{prop_computable}
If~$\Y$ admits a computable fixed-point free multifunction then there is no computable function from~$\Baire$ to~$\C(\Baire,\Y)$ containing all the computable elements in its range.
\end{proposition}
\begin{proof}
In the proof of Proposition \ref{prop_continuous}, if~$H$ and~$\Phi$ are computable, then the function~$F$ is computable and is not in the range of~$\phi$.
\end{proof}

Proposition \ref{prop_computable} implies that the computable functions from~$\Baire$ to~$\Y$ are not~$\Sigma^0_1$-enumerable, as the subsets of~$\N$ that are computable images of~$\Baire$ are the~$\Sigma^0_1$-sets. However one would expect the much stronger result that they are not~$\Sigma^1_1$-enumerable. It is possible to obtain such a result under a mild assumption on~$\delta_Y$, usually satisfied in practice.

\begin{definition}
We say that~$A\subseteq\Baire$ is \defin{computably separable} if~$A$ contains a computable sequence that is dense in~$A$.
\end{definition}
If~$A$ is computably separable then the set of cylinders intersecting~$A$ is c.e., and there is a computable function mapping any cylinder intersecting~$A$ to an element of~$A$ in that cylinder. This property has the following interesting consequence.

\begin{lemma}\label{lemma_extension}
Assume that~$\dom(\delta_Y)$ is computably separable. Every computable function from a~$\Pi^0_1$-subset of~$\Baire$ to~$\Y$ has a total computable extension.
\end{lemma}
\begin{proof}
Let~$P$ be a~$\Pi^0_1$-subset of~$\Baire$,~$\phi:P\to \Y$ be computable and~$\Phi:P\to\Baire$ be a computable realizer for~$\phi$. We define a total computable extension~$\phi'$ of~$\phi$ be defining a computable realizer of~$\phi'$ as follows. Given~$f$, run two algorithms in parallel. On the one hand, start computing~$\Phi(f)$ and output the result as long as it is compatible with (i.e.\ has an extension in)~$\dom(\delta_Y)$. On the other hand, test whether~$f\notin P$. If one eventually discovers that~$f\notin P$, then extend the current output with some canonical element of~$\dom(\delta_Y)$.
\end{proof}


We now get the effective version of Proposition \ref{prop_continuous}.
\begin{proposition}\label{prop_pi01}
Assume that~$\dom(\delta_Y)$ is computably separable and~$\Y$ admits a computable fixed-point free multifunction. The computable functions from~$\Baire$ to~$\Y$ are not~$\Sigma^1_1$-enumerable.
\end{proposition}
\begin{proof}
Every~$\Sigma^1_1$-subset of~$\N$ is a computable image of a~$\Pi^0_1$-subset of~$\Baire$, so it is sufficient to prove that the computable elements of~$\C(\Baire,\Y)$ is not a computable image of such a set.

Let~$P$ be a~$\Pi^0_1$-subset of~$\Baire$ and~$\delta:P\to\C(\Baire,\Y)$ be computable. The computable separability of the domain of~$\delta_Y$ enables one to extend~$\delta$ outside~$P$. Indeed,~$\delta$ is equivalently seen as a computable function from~$P\times \Baire$ to~$\Y$, and~$P\times\Baire$ is a~$\Pi^0_1$-set, so Lemma \ref{lemma_extension} can be applied. We now apply Proposition \ref{prop_computable}.
\end{proof}

\begin{example}
There is a strong contrast between the two following cases:
\begin{itemize}
\item The left-c.e.\ functions from~$\Baire$ to~$[0,1]$ can be effectively enumerated (from~$\N$). Indeed, if~$q_n$ is a computable enumeration of the rational numbers in~$[0,1]$ then from any computable sequence~$(U_n)_{n\in\N}$ of open subsets of~$\Baire$ (those sequences can be effectively enumerated), one can build the left-c.e.\ function~$F(f)=\sup\{q_n:f\in U_n\}$ (with~$\sup\emptyset=0$).

\item However the left-c.e.\ functions from~$\Baire$ to~$[0,1)$ are not~$\Sigma^1_1$-enumerable. Indeed, the computable function~$h:[0,1)_<\to[0,1)_<$ mapping~$x$ to~$(1+x)/2>x$ is fixed-point free, and the representation of~$[0,1)_<$ has a computably separable domain (a real number~$x\in[0,1)_<$ is represented by a non-decreasing sequence of rationals converging to~$x$).
\end{itemize}
\end{example}




\paragraph{Oracles.}
So far we have measured the complexity of parametrizing points by measuring the descriptive complexity of the domain of a computable enumeration. A different but related approach is to measure the complexity of an oracle computing an enumeration.

\begin{corollary}
Assume that~$\dom(\delta_Y)$ is computably separable and~$\Y$ admits a computable fixed-point free multifunction. There is no~$\Delta^1_1$-oracle enumerating a list of indices of all the computable points of~$\C(\Baire,\Y)$.
\end{corollary}
\begin{proof}
If there is such an oracle then the set of enumerated indices is also a~$\Delta^1_1$-subset of~$\N$, and the function mapping an index in that set to the corresponding point is a computable onto function with a~$\Delta^1_1$ domain, contradicting Proposition \ref{prop_pi01}.
\end{proof}

However it may happen that a~$\Delta_1^1$-oracle enumerates all the computable points, producing names rather than indices. Again this is not possible under the same assumptions.

\begin{proposition}\label{prop_oracle}
Assume that~$\dom(\delta_Y)$ is computably separable and~$\Y$ admits a computable fixed-point free multifunction. There is no~$\Delta^1_1$-oracle enumerating all the computable points of~$\C(\Baire,Y)$.
\end{proposition}
\begin{proof}
Let~$A\in\Delta^1_1(\N)$ and~$\phi:\N\to\C(\Baire,\Y)$ be computable with oracle~$A$. The set~$\{A\}$ is a~$\Delta^1_1$-subset (hence a~$\Sigma^1_1$-subset) of~$\Baire$ so it is a computable image of a~$\Pi^0_1$-subset~$P$ of~$\Baire$. As a result, the function~$\phi':P\times\N\to\C(\Baire,\Y)$ mapping~$(f,n)$ to~$\phi(n)$ is a computable function defined on a~$\Pi^0_1$-subset of~$\Baire\times\N\cong\Baire$. As in the proof of Propostion \ref{prop_pi01},~$\phi'$ has a total computable extension, which cannot contain all the computable elements in its range.
%
%
\end{proof}
\section{Applications to open sets}\label{sec_applications}
Let~$\X$ be a represented space.~$\Y:=\O(\X)$ is a represented space with a computably separable representation (every cylinder contains a representation of~$X\in\O(\X)$ consisting of a covering of~$\Baire$, and that can be computed from the cylinder). In \cite{deBSS16}, a represented space~$\X$ is called~$\Baire$-based if there is a continuous function from~$\Baire$ to~$\O(\X)$ whose image is a basis of the topology on~$\X$. This is equivalent to saying that~$\O(\X)$ itself is a continuous image of~$\Baire$.

\begin{lemma}\label{lem_baire_based}
A represented space~$\X$ is~$\Baire$-based iff~$\O(\X)$ is a continuous image of~$\Baire$.
\end{lemma}
\begin{proof}
The space~$\X$ is hereditarily Lindel\"of, i.e.\ every open cover contains a countable subcover (simply because~$\Baire$ and its subspaces satisfy this property and taking the final topology of the representation preserves this property). Now assume that~$\X$ is~$\Baire$-based and let~$\Phi:\Baire\to\O(\X)$ be continuous, such that~$\Phi[\Baire]$ is a basis. We can easily modify~$\Phi$ so that the empty set is in its image. A function~$f:\N\to\N$ can equivalently be seen as a sequence of functions~$f_i:\N\to\N$. Let~$\Psi(f)=\bigcup_i\Phi(f_i)$. The function~$\Psi:\Baire\to\O(\X)$ is continuous and onto, as every open subset of~$\O(\X)$ is a countable union of elements of the basis.
\end{proof}

Observe that~$\C(\Baire,\O(\X))\cong\O(\Baire\times \X)$, hence Propositions \ref{prop_continuous}, \ref{prop_pi01}, and \ref{prop_oracle} immediately give the following results.

\begin{theorem}\label{thm_OX}
\begin{enumerate}
\item If~$\O(\X)$ admits a continuous fixed-point free multifunction then~$\Baire\times \X$ is not~$\Baire$-based.
\item If~$\O(\X)$ admits a computable fixed-point free multifunction then the computable open subsets of~$\Baire\times \X$:
\begin{itemize}
\item Are not~$\Sigma^1_1$-enumerable,
\item Are not enumerable relative to any~$\Delta_1^1$ oracle.
\end{itemize}
\end{enumerate}
\end{theorem}


Observe that by the Kleene fixed-point theorem, every continuous function from~$\O(\X)$ to~$\O(\X)$ has a fixed-point, so we really need to consider multifunctions rather than functions.

Also note that if~$\X$ is a countably-based space with its standard representation then~$\O(\X)$ has no continuous fixed-point free multifunction, because~$\Baire\times\X$ is countably-based hence~$\O(\Baire\times\X)$ is a continuous imageof~$\Baire$.

\subsection{Product space}\label{sec_product}
Let us consider again the admissibly represented space~$\X=\Baire\times \O(\Baire)$. Its topology is not the product topology and contrary to the product topology it is not~$\Baire$-based as shown below.

\begin{lemma}\label{lemma_OOBaire}
The space~$\O(\O(\Baire))$ admits a computable fixed-point free multifunction.
\end{lemma}
\begin{proof}
Given a name of~$\U\in\O(\O(\Baire))$, we build a function~$g:\N\to\N$ and define~$\V=\{O\in\O(\Baire):g\in O\}$. The function~$g$ is defined as follows.
Start describing the function~$g$ mapping everything to~$0$. If eventually one finds that~$\U\neq\emptyset$ then do as follows. So far,~$g$ belongs to some cylinder~$[0^n]$. As~$\U\neq\emptyset$, ~$\Baire\in\U$ so any enumeration of~$\Baire$ is accepted by the name of~$\U$. Enumerate~$\Baire$ in such a way that~$[0^n]$ is not covered in finite time. When this enumeration is accepted by~$\U$, let~$g$ be an element of~$[0^n]$ that is not yet covered by the enumeration of~$\Baire$.

If~$\U=\emptyset$ then~$g=0$ and~$\V$ is non-empty so~$\V\neq \U$. If~$\U\neq \emptyset$ then~$\Baire\setminus\{g\}$ belongs to~$\U$ but not to~$\V$, so~$\V\neq\U$.
\end{proof}

One can apply Theorem \ref{thm_OX}.
\begin{itemize}
\item $\Baire\times\O(\Baire)$ is not~$\Baire$-based,
\item The computable open subsets of~$\Baire\times \O(\Baire)$ are neither~$\Sigma^1_1$-enumerable nor enumerable relative to a~$\Delta^1_1$ oracle. 
\end{itemize}

However,
\begin{proposition}
$\Baire\times\O(\Baire)$ is~$\NN$-based.
\end{proposition}
\begin{proof}
$\O(\O(\Baire))$ admits a total representation~$\delta:\Baire\to\O(\O(\Baire))$, equivalent to the canonical representation. Therefore the function~$\C(\Baire,\Baire)\to\C(\Baire,\O(\O(\Baire))$ mapping~$f$ to~$\delta\circ f$ is continuous and onto. Observe that~$\C(\Baire,\Baire)\cong\NN$ and~$\C(\Baire,\O(\O(\Baire))\cong\O(\Baire\times\O(\Baire))$.
\end{proof}

Moreover~$\O(\Baire\times\O(\Baire))$ is a computable image of a~$\Pi^1_1$-subset of~$\Baire$, so the computable open subsets of~$\Baire\times\O(\Baire)$ are~$\Pi^1_1$-enumerable and enumerable relative to a~$\Delta^1_2$-oracle.

\subsection{Space of clopens}
Let~$\X=\C(\Baire,2)$ be the space of clopen subsets of the Baire space.

\begin{theorem}\label{thm_CNN}
The space~$\C(\Baire,2)$ is not~$\Baire$-based. Its computable open sets are not~$\Sigma^1_1$-enumerable and are not enumerable relative to any~$\Delta^1_1$ oracle.
\end{theorem}
To prove the theorem we need the following lemma.

\begin{lemma}\label{lem_NN_fpf}
The space~$\O(\C(\Baire,2))$ admits a computable fixed-point free multifunction.
\end{lemma}
The proof is almost identical to the proof of Lemma \ref{lemma_OOBaire}.
\begin{proof}
Given a name of~$\U\in\O(\C(\Baire,2))$, start computing a function~$g\in\Baire$ as explained below, and output~$\V=\{C\in \C(\Baire,2):g\in C\}\in \O(\C(\Baire,2))$.

Start defining~$g=0$ on the first inputs. At the same time, wait to see whether~$\U$ is non-empty. If eventually~$\U\neq\emptyset$ then at that point~$g$ is defined up to input~$n$. One can find some~$C\in \U$ and~$[0^n]\nsubseteq C$. Indeed, start from some~$A\in\U$ and generate a name of~$A$ which does not cover~$[0^n]$ in finite time. This name is eventually accepted by the name of~$\U$. Once~$C$ is found, let~$g\in[0^n]\setminus C$.

We now check that~$\V\neq\U$. If~$\U$ is empty then~$\V\neq \U$ as~$\V$ is non-empty (and in that case,~$g=0$). If~$\U$ is non-empty then one finds~$C\in\U$ and~$g\notin C$, so~$C\notin \V$ hence~$\V\neq\U$.
\end{proof}

\begin{proof}[Proof of Theorem \ref{thm_CNN}]
Let~$\X=\C(\Baire,2)$. Lemma \ref{lem_NN_fpf} and Theorem \ref{thm_OX} imply the result for the space~$\Baire\times\X$.

We know from Lemma \ref{lem_retract_NN_NN2} that~$\Baire$ is a computable retract of~$\X$, so~$\Baire\times \X$ is a computable retract of~$\X\times\X\cong\X$. As a result Lemma \ref{lem_retract_open} implies that~$\O(\Baire\times \X)$ is a computable retract, hence a computable image of~$\O(\X)$, so the negative results about~$\O(\Baire\times \X)$ immediately apply to~$\O(\X)$.
\end{proof}

Observe that Theorem \ref{thm_CNN} immediately applies to the space~$\NN=\C(\Baire,\N)$, as~$\C(\Baire,2)$ is a computable retract of~$\NN$.

We will now prove the stronger result that~$\C(\Baire,2)$ and~$\NN$ are not even~$\NN$-based (Theorem \ref{thm_CNN2} below), and more generally~$\KK{k}$ is not~$\KK{k}$-based for~$k\geq 2$. For this, we need to refine our techniques, as we explain now.

\section{Canonicity of names}\label{sec_canon}
The proof of Proposition \ref{prop_continuous} works because~$\Baire$ has a particular property: every element~$f\in\Baire$ has a canonical name (which is~$f$ itself). This property has the consequence that every continuous (resp.\ computable) multifunction~$F:\Baire\tto\X$ has a continuous (resp.\ computable) choice function~$F':\Baire\to\X$, such that~$F'(f)\in F(f)$ for all~$f\in\Baire$. So the multifunction built using the diagonal argument and the fixed-point free multifunction is actually a single-valued function.

The spaces having this canonicity property are exactly the subspaces of~$\Baire$. The proof of Proposition \ref{prop_continuous} does not carry over to spaces without this canonicity property, like~$\R$ or~$\NN$ for instance. However we provide a technique to overcome this problem.

\begin{definition}\label{def_canon}
A represented space~$\X$ is \defin{canonizable} if there exists a subset~$P\subseteq\X$ such that:
\begin{itemize}
\item $P$ is closed,
\item $\X$ is a continuous image of~$P$,
\item $P$ is homeomorphic to some subspace of~$\Baire$.
\end{itemize}
\end{definition}

The first property implies that~$P$ as a represented subspace of~$\X$ has the subspace topology. The third property exactly says that~$P$ has the canonicity property, i.e.\ that there is a continuous function mapping every~$x\in P$ to a name of~$x$. 

\begin{example}
$\R$ is canonizable, because it is a continuous image of~$\N\times \cantor$ which embeds as a closed set in~$\R$.
\end{example}

When a space is canonizable, the diagonal argument can be applied.
\begin{theorem}\label{thm_canon}
Let~$\X$ be a canonizable space. If~$\O(\Y)$ admits a continuous fixed-point free multifunction then~$\O(\X\times \Y)$ is not a continuous image of~$\X$.
\end{theorem}
\begin{proof}
Let~$P$ make~$\X$ canonizable, and let~$Q\subseteq\Baire$ be homeomorphic to~$P$. Proposition \ref{prop_continuous} implies that~$\O(Q\times\Y)\cong \C(Q,\O(\Y))$ is not a continuous image of~$Q$, hence~$\O(P\times\Y)$ is not a continuous image of~$P$. As~$P$ is a closed subset of~$\X$,~$\O(P\times \Y)$ is a continuous retract, hence a continuous image of~$\O(\X\times\Y)$. As~$\X$ is a continuous image of~$P$, 
one concludes by transitivity that~$\O(\X\times\Y)$ is not a continuous image of~$\X$. 
\end{proof}

This result has a computable version. We say that~$\X$ is \defin{computably canonizable} if the set~$P$ is a~$\Pi^0_1$-subset of~$\X$, if~$\X$ is a computable image of~$P$ and~$P$ is computably homeomorphic to some subspace of~$\Baire$. In that case, the same proof shows that there is no computable function from~$\X$ to~$\O(\X\times\Y)$ containing all the computable open sets in its range.

\subsection{Kleene-Kreisel functionals as parameters}

The notion of a canonizable space enables us to extend Theorem \ref{thm_OX} from~$\Baire=\KK{1}$ to the whole hierachies functionals of finite and countable type.

Observe that for~$\alpha\geq 1$ a space~$\X$ is~$\KK{\alpha}$-based if and only if~$\O(\X)$ is a continuous image of~$\KK{\alpha}$, because as in the proof of Lemma \ref{lem_baire_based}, a continuous function from~$\KK{\alpha}$ to a basis of the topology on~$\X$ can be extended to a continuous surjection from~$\C(\N,\KK{\alpha})\cong\KK{\alpha}$ to~$\O(\X)$.

\begin{theorem}\label{thm_KKalpha}
\begin{enumerate}
\item If~$\O(\X)$ admits a continuous fixed-point free multifunction then for each countable ordinal~$\alpha$,~$\KK{\alpha}\times \X$ is not~$\KK{\alpha}$-based.
\item If~$\O(\X)$ admits a computable fixed-point free multifunction then for each~$k\in\N$ the computable open subsets of~$\KK{k}\times \X$:
\begin{itemize}
\item Are not~$\Sigma^1_k$-enumerable,
\item Are not enumerable relative to any~$\Delta^1_k$ oracle.
\end{itemize}
\end{enumerate}
\end{theorem}

The proof essentially follows the same argument as the proof of Theorem \ref{thm_OX}, but we first need to show that the spaces~$\KK{\alpha}$ are canonizable.

\begin{lemma}\label{lem_kkk_canon}
For every countable ordinal~$\alpha$,~$\KK{\alpha}$ is canonizable.
\end{lemma}
\begin{proof}
The proof is inspired by a construction appearing in the proof of Theorem 7.2 in \cite{SchroderS15a}.

We first prove the result for all successor ordinals. $D_{\alpha+1}$ is a~$\PPi^1_\alpha$-set, so there exists a continuous surjective function~$f:\KK{\alpha}\to \Baire\setminus D_{\alpha+1}$. We define~$\psi:D_{\alpha+1}\to\KK{\alpha+1}$ as follows. For~$x\in D_{\alpha+1}$ and~$y\in\KK{\alpha}$, let~$\psi(x)(y)=\min\{n\in\N:x_n\neq f(y)_n\}$. $\psi$ is continuous, let~$\Psi:D_{\alpha+1}\to\Baire$ be a continuous realizer of~$\psi$. Let~$P\subseteq\Baire\times\KK{\alpha+1}$  be the graph of~$\psi$. $P$ is homeomorphic to the graph of~$\Psi$ which is a subset of~$\Baire\times\Baire\cong\Baire$. $P$ is a closed subset of~$\Baire\times\KK{\alpha+1}$. Indeed, a pair~$(x,F)$ belongs to~$P$ iff for all~$y\in\KK{\alpha}$, if~$n=F(y)$ then~$x_n\neq f(y)_n$ and~$x_m=f(y)_m$ for all~$m<n$, which can be checked for~$y$ in some countable dense subset of~$\KK{\alpha}$. If this is not true then one can eventually see it. Finally,~$\KK{\alpha+1}$ is a continuous image of~$D_{\alpha+1}$ which is the first projection of~$P$.

For limit ordinals, we prove the result by induction. Let~$\lambda$ be a limit ordinal. For each each~$\beta<\lambda$, let~$P_\beta$ witness that~$\KK{\beta}$ is canonizable. By definition,~$\KK{\lambda}=\prod_{i\in\omega}\KK{\beta(\lambda,i)}$ for some numbering~$\beta(\lambda,i)$ of the ordinals smaller than~$\lambda$, and one can easily check that~$\prod_{i\in\omega}P_{\beta(\lambda,i)}$ makes~$\KK{\lambda}$ canonizable.
\end{proof}

For the finite levels of the hierarchy we can prove an effective version.
\begin{lemma}\label{lem_kk_canon}
For~$k\geq 1$,~$\KK{k}$ is computably canonizable.
\end{lemma}

\begin{proof}
From Theorem \ref{thm_kk_sigma}, as~$D_{k}\in\Pi^1_{k-1}$ there exists a computable predicate~$R\subseteq\Baire\times \KK{k-1}\times\N$ such that~$x\in D_{k}\iff\forall y\in\KK{k-1},\exists n, R(x,y,n)$. We define a computable function define~$\psi:D_{k}\to\KK{k}$ by~$\psi(x)(y)=\min\{n:R(x,y,n)\}$ for~$x\in D_k, y\in\KK{k-1}$. Let~$P$ be the graph of~$\psi$. First,~$P$ is a~$\Pi^0_1$-subset of~$\Baire\times\KK{k}$. Indeed,~$(x,H)$ belongs to the graph if and only if for all~$y\in\KK{k-1}$, if~$n=H(y)$ then~$R(x,y,n)$ and~$\neg R(x,y,i)$ for all~$i<n$, but this can be checked for all~$y$ in a computable dense sequence in~$\KK{k-1}$ (which exists by Lemma \ref{lem_kk_dense}).

As~$\psi$ is computable, its graph is computably isomorphic to the graph of a computable realizer of~$\psi$, which is a subset of~$\Baire\times\Baire\cong\Baire$. Finally,~$\KK{k}$ is a computable image of~$D_k$ which is the first projection of~$P$.
\end{proof}

It may possible to extend the result to the constructive ordinals, however we do not address this question in this article.

\begin{proof}[Proof of Theorem \ref{thm_KKalpha}]
In the first case, as~$\KK{\alpha}$ is canonizable we can apply Theorem \ref{thm_canon}.

In the second case, as~$\KK{k}$ is computably canonizable there is no computable function from~$\KK{k}$ to~$\O(\KK{k}\times\Y)$ containing all the computable open sets in its range, therefore there is no such computable function defined on some~$\Pi^0_1$-subset of~$\KK{\alpha}$. We can conclude as every set in~$\Sigma^1_k(\N)$ is a computable image of some set in~$\Pi^0_1(\KK{\alpha})$, and for every set~$A\in\Delta^1_2(\N)$,~$\{A\}$ is a computable image of some set in~$\Pi^0_1(\KK{\alpha})$.
\end{proof}

\subsection{Applications}
We now apply Theorem \ref{thm_KKalpha} to particular spaces.
\paragraph{Back to the space of clopens.}
We saw in Theorem \ref{thm_CNN} that~$\C(\Baire,2)$ is not~$\Baire$-based, i.e.\ that~$\O(\C(\Baire,2))$ is not a continuous image of~$\Baire$. We can now apply Theorem \ref{thm_KKalpha} to improve that result.
\begin{theorem}\label{thm_CNN2}
The space~$\C(\Baire,2)$ is not~$\NN$-based.

The computable open subsets of~$\C(\Baire,2)$ are not~$\Sigma^1_2$-enumerable and not enumerable relative to any~$\Delta^1_2$ oracle.
\end{theorem}

\paragraph{The Kleene-Kreisel functionals.}

\begin{theorem}\label{thm_kkk}
For every countable ordinal~$\alpha\geq 2$,~$\KK{\alpha}$ is not~$\KK{\alpha}$-based.

For every~$k\geq 2$, the computable open subsets of~$\KK{k}$ are not~$\Sigma^1_k$-enumerable and not enumerable relative to any~$\Delta^1_k$ oracle.
\end{theorem}

We just have to prove that~$\O(\KK{\alpha})$ has continuous fixed-point free multifunction and~$\O(\KK{k})$ has a computable one. They can be easily derived from the computable fixed-point free multifunction on~$\O(\C(\Baire,2))$ (Lemma \ref{lem_NN_fpf}) thanks to the following observation.

\begin{lemma}\label{lem_retract_fpf}
Let~$\X$ be a continuous (resp.\ computable) retract of~$\Y$. If~$\X$ admits a continuous (resp.\ computable) fixed-point free multifunction then so does~$\Y$.
\end{lemma}
\begin{proof}
Let~$r:\Y\to\X$ and~$s:\X\to\Y$ be continuous such that~$r\circ s =\id$, and let~$R,S$ be continuous realizers. Let~$h:\X\tto\X$ be fixed-point free and~$H$ be a continuous realizer. The multifunction~$k=s\circ h\circ r$ (defined by~$k(y)=\{s(x):x\in h(r(y)\}$) is continuous as it is realized by the continuous function~$S\circ H\circ R$. Moreover it is fixed-point free, as if~$y\in k(y)$ then~$y=s(x)$ for some~$x\in h(r(y))$, so~$r(y)$ is a fixed-point for~$h$, as~$r(y)=r\circ s(x)=s\in h(r(y))$, contradicting the assumption about~$h$. Everything works the same in the computable case.
\end{proof}

Theorem \ref{thm_kkk} follows as~$\C(\Baire,2)$ is a computable retract of~$\KK{2}$ which is a computable retract of each~$\KK{k}$ (Lemma 7.4 in \cite{SchroderS15a}),~$k\geq 2$, and a continuous retract of each~$\KK{\alpha}$,~$\alpha\geq 2$.

\paragraph{Implications on the hyperprojective base-hierarchy.}

In \cite{deBSS16} the hierachy~$\Based(\KK{\alpha})$ is introduced: a space is in~$\Based(\KK{\alpha})$ if it is~$\KK{\alpha}$-based. That hierarchy is related to another hierarchy called~$\QCB{\alpha}$. The following results are proved in \cite{deBSS16}:
\begin{itemize}
\item Proposition 2.9: $\KK{\alpha+1}\in\QCB{\alpha}$,
\item Proposition 6.11: $\QCB{\alpha}\subseteq \Based(\KK{\alpha+2})$,
\item Corollary 6.12: The inclusion~$\Based(\KK{\alpha})\subset\Based(\KK{\alpha+3})$ is proper.
\end{itemize}

It is asked:
\begin{itemize}
\item Problem 6.2: for which~$\alpha$ does the inclusions $\QCB{\alpha}\subseteq \Based(\KK{\alpha+1})$ and~$\QCB{\alpha}\subseteq \Based(\KK{\alpha})$ hold?
\item Is the inclusion~$\Based(\KK{\alpha})\subset\Based(\KK{\alpha+1})$ proper?
\end{itemize}

Theorem \ref{thm_kkk} gives answers to these questions:
\begin{corollary}
\begin{enumerate}
\item For all countable ordinals~$\alpha\geq 1$,~$\QCB{\alpha}\nsubseteq \Based(\KK{\alpha+1})$.
\item For all countable ordinals~$\alpha$, the inclusion~$\Based(\KK{\alpha})\subset\Based(\KK{\alpha+1})$ is proper (witnessed by~$\KK{\alpha}$),
\end{enumerate}
\end{corollary}
\begin{proof}
\begin{enumerate}
\item The first assertion is witnessed by~$\KK{\alpha+1}$ (Theorem \ref{thm_kkk}).
\item The second assertion is witnessed by~$\KK{\alpha}$ for~$\alpha\geq 2$ (Theorem \ref{thm_kkk}), by~$\Baire\times\O(\Baire)$ for~$\alpha=1$ (see Section \ref{sec_product}) and by~$\O(\Baire)$ for~$\alpha=0$.\qedhere
\end{enumerate}
\end{proof}

\subsection{Markov computability}\label{sec_markov}
Let~$\X_M$ be the set of computable elements of~$\X$ with the Markov representation mapping an index of a computable name of~$x\in\X$ to~$x$ (an index~$i\in\N$ can be identified with the constant function with value~$i$, so that the domain of the representation is a subset of~$\Baire$). Note that the identity from~$\X_M$ to~$\X_c$ is computable, but its inverse is not computable in general.

While the spaces~$\Baire$ and~$\Baire_c$ have the canonicity property,~$\Baire_M$ does not: it is not possible to associate to each computable function a canonical index that could be uniformly computed from any other index.

However,

\begin{lemma}
$\Baire_M$ is computably canonizable.
\end{lemma}
\begin{proof}
Let~$(\varphi_i)_{i\in\N}$ be a canonical enumeration of the partial computable functions from~$\N$ to~$\N$. Let~$\Tot=\{i:\varphi_i\text{ is total}\}$. If~$i\in\Tot$, then let~$t_i:\N\to\N$ map~$n$ to the halting time of~$\varphi_i(n)$.

Let~$P=\{(i,t):\varphi_i\text{ is total and }t=t_i\}$, where~$(i,t)$ is understood as the function~$f$ mapping~$0$ to~$i$ and~$n+1$ to~$t(n)$. $P$ is a~$\Pi^0_1$-subset of~$\Baire_M$ (and even of~$\Baire_c$). Indeed, given~$f$, let~$i=f(0)$: one has~$f\notin P\iff\exists n,\varphi_i(n)$ does not halt in exactly~$f(n)$ steps, which can be eventually discovered.

On~$P$, the standard representation and the Markov representation are computably equivalent: given~$f\in P$, one can compute an index of~$f$, as one can compute an index of~$t_i$ where~$i:=f(0)$, so~$P\subseteq\X_M$ is computably isomorphic to~$P\subseteq\Baire$.

Finally, $\Baire_M$ is a computable image of the first projection of~$P$.
\end{proof}

\begin{lemma}\label{lem_M_fpf}
$\O(\Baire_M)$ admits a computable fixed-point free multifunction.
\end{lemma}
\begin{proof}
One can think of a name of an element~$\W$ of~$\O(\Baire_M)$ as an enumeration of a set~$W\subseteq \N$ containing the indices of the computable functions in~$\W$ (and possibly indices of partial functions). Given an enumeration of a set of indices~$W$, we enumerate another set~$W'$ coding some~$\W'\in\O(\Baire_M)$ such that~$\W'\neq\W$.

Let~$f_\infty$ be the null function, and~$f_k(n)=0$ if~$n\neq k$,~$f_k(k)=1$.

Given an enumeration of~$W$, we enumerate~$W'$ as follows. Start enumerating all the natural numbers. In parallel, test whether~$f_\infty\in \W$. If it eventually happens, let~$n$ be larger than all the numbers enumerated so far in~$W'$. Look for~$n+1$ values of~$k$ such that~$f_k\in\W$, and let~$K$ be larger than all of them. Now enumerate in~$W'$ all the numbers~$j$ such that~$\varphi_j$ coincides with some~$\varphi_l$,~$l<n$, on inputs~$0,1,\ldots,K$.

Now we check that the construction fulfills the conditions.
\begin{itemize}
\item If~$f_\infty\notin \W$ then~$W'=\N$ and~$\W'=\Bairec$, so~$f_\infty\in \W'\setminus\W$.
\item If~$f_\infty\in \W$ then there exist infinitely many~$k$'s such that~$f_k\in\W$, so the algorithm will eventually find~$K$, and~$\W'=\bigcup_{l<n}[\varphi_l\res{K}]$ (with the convention that $[\varphi_l\res{K}]=\emptyset$ if~$\varphi_l$ is not defined up to~$K$). As there are at least~$n+2$ values of~$k< K$ such that~$f_k\in\W$, and all these functions differ on some input~$<K$, one of them does not belong to~$\W'$ which a union of~$n$ cylinders of length~$K$. As result,~$f_k\in \W\setminus\W'$.\qedhere
\end{itemize}
\end{proof}

The computable elements of~$\O(\Baire_M)$ are also called the Markov semidecidable properties. They correspond to the c.e.\ subsets of~$\N$ that are extensional w.r.t.\ the indices of total computable functions (for each total function, an extensional set contains all its indices or none of them). With the tools we have developed, we can resolve a question left open in \cite{ICALP16}.
\begin{corollary}\label{cor_markov}
The computable open subsets of~$\Baire_M$ are not~$\Sigma^0_3$-enumerable.
\end{corollary}
\begin{proof}
One has~$\O(\Baire_M\times\Baire_M)\cong\O(\Baire_M)$ so~$\O(\Baire_M)$ is not a computable image of~$\Baire_M$. The computable images of~$\Baire_M$ are exactly the~$\Sigma^0_3$-sets.
\end{proof}

Finally, the computable open subsets of~$\Baire_M$ are~$\Pi^0_3$-enumerable, simply because the domain of the representation of~$\O(\Baire_M)$ is~$\Pi^0_3$ (said differently, for a c.e.\ set, being extensional is a~$\Pi^0_3$-property).

\section{Discussion and results}
The main results presented in this article are:
\begin{itemize}
\item There exists of a~$\Sigma^0_2$-subset of~$\O(\Baire)$ that is not Borel (Theorem \ref{thm_sigma2}),
\item For each countable ordinal~$\alpha\geq 2$, there is no continuous surjection from~$\KK{\alpha}$ to~$\O(\KK{\alpha})$ or, equivalently, there is no~$\KK{\alpha}$-universal open set for~$\KK{\alpha}$ (Theorem \ref{thm_kkk}),
\item For~$k\in\N$,~$k\geq 2$, there is no computable enumeration of the computable open subsets of~$\KK{k}$ from a~$\Sigma^1_k$-subset of~$\N$, or relative to a~$\Delta^1_k$ oracle (Theorem \ref{thm_kkk}),
\item There is no computable enumeration of the Markov semidecidable subsets of~$\Baire$ from any~$\Sigma^0_3$-subset of~$\N$  (Corollary \ref{cor_markov}).
\end{itemize}

These results are obtained by building fixed-point free multi-valued functions and applying the diagonal argument. In several cases one has to be careful about extensionality issues, and we introduce the notion of a canonizable space (Definition \ref{def_canon}) to overcome this problem. These techniques might be useful for other investigations.

These results are negative results that help locating the complexity of describing sets in some pointclasses. An interesting direction would be to obtain positive answers to questions like: What do the~$\SSigma^0_2$-subsets and~$\Sigma^0_2$-subsets of~$\O(\Baire)$ look like? What do the open subsets of~$\N^{\N^\N}$ look like? How to build them from simpler objects?

\section*{Acknowledgments}
The author would like to thank Arno Pauly for suggesting the investigation of the~$\Sigma^0_2$-subsets of~$\O(\Baire)$ and the $\Sigma^0_1$-subsets of~$\C(\Baire,2)$, as well as Matthias Schr\"oder for discussions on this topic.

%

\begin{thebibliography}{dBSS16}

\bibitem[dB13]{Brecht13}
Matthew de~Brecht.
\newblock Quasi-polish spaces.
\newblock {\em Ann. Pure Appl. Logic}, 164(3):356--381, 2013.

\bibitem[dBSS16]{deBSS16}
Matthew de~Brecht, Matthias Schr{\"o}der, and Victor Selivanov.
\newblock Base-complexity classifications of qcb{$_0$}-spaces.
\newblock {\em Computability}, 5(1):75--102, 2016.

\bibitem[EH02]{EscardoH01}
Martin Escard{\'o} and Reinhold Heckmann.
\newblock Topologies on spaces of continuous functions.
\newblock {\em Topology Proceedings}, 26(2):545--564, 2001-2002.

\bibitem[Hoy16]{ICALP16}
Mathieu Hoyrup.
\newblock The decidable properties of subrecursive functions.
\newblock In {\em {ICALP} 2016, July 11-15, 2016, Rome, Italy}, volume~55 of
  {\em LIPIcs}, pages 108:1--108:13. Schloss Dagstuhl - Leibniz-Zentrum fuer
  Informatik, 2016.

\bibitem[Kec95]{Kechris95}
Alexander~S. Kechris.
\newblock {\em {Classical Descriptive Set Theory}}.
\newblock Springer, January 1995.

\bibitem[Kle59]{kleene1959countable}
S.C. Kleene.
\newblock Countable functionals.
\newblock {\em Constructivity in Mathematics}, pages 81--100, 1959.

\bibitem[Kre59]{Kreisel59}
G.~Kreisel.
\newblock Interpretation of analysis by means of functionals of finite type.
\newblock {\em Constructivity in Mathematics}, pages 101--128, 1959.

\bibitem[Law69]{Lawvere69}
F.~William Lawvere.
\newblock {\em Diagonal arguments and cartesian closed categories}, pages
  134--145.
\newblock Springer Berlin Heidelberg, Berlin, Heidelberg, 1969.

\bibitem[Nor80]{Normann80}
D.~Normann.
\newblock {\em Recursion on the Countable Functionals}.
\newblock Lecture Notes in Mathematics. Springer Berlin Heidelberg, 1980.

\bibitem[NV97]{NorbergV97}
Tommy Norberg and Wim Vervaat.
\newblock Capacities on non-{H}ausdorff spaces.
\newblock {\em Probability and Lattices}, 110:133--150, 1997.

\bibitem[Pau15]{Pauly15}
Arno Pauly.
\newblock On the topological aspects of the theory of represented spaces.
\newblock {\em Computability}, 5(2):159--180, 2015.

\bibitem[PdB15]{PaulyB15}
Arno Pauly and Matthew de~Brecht.
\newblock Descriptive set theory in the category of represented spaces.
\newblock In {\em 30th Annual {ACM/IEEE} Symposium on Logic in Computer
  Science, {LICS} 2015, Kyoto, Japan, July 6-10, 2015}, pages 438--449. {IEEE}
  Computer Society, 2015.

\bibitem[Sel04]{Selivanov04}
V.~L. Selivanov.
\newblock Difference hierarchy in {$\varphi$}-spaces.
\newblock {\em Algebra and Logic}, 43(4):238--248, Jul 2004.

\bibitem[SS15a]{SchroderS15a}
Matthias Schr{\"o}der and Victor Selivanov.
\newblock Some hierarchies of qcb\({}_{\mbox{0}}\)-spaces.
\newblock {\em Mathematical Structures in Computer Science}, 25(8):1799--1823,
  2015.

\bibitem[SS15b]{SchroderS15b}
Matthias Schr{\"{o}}der and Victor~L. Selivanov.
\newblock Hyperprojective hierarchy of qcb\({}_{\mbox{0}}\)-spaces.
\newblock {\em Computability}, 4(1):1--17, 2015.

\bibitem[Wei00]{Wei00}
Klaus Weihrauch.
\newblock {\em Computable Analysis}.
\newblock Springer, Berlin, 2000.

\end{thebibliography}
\end{document}